\documentclass[11pt,a4paper,twoside]{article}
\usepackage[UKenglish]{babel}
\usepackage[T1]{fontenc}
\usepackage[utf8x]{inputenc}
\usepackage{latexsym,amsfonts,amsmath,amsthm,amssymb}
\usepackage{fullpage}
\usepackage{xcolor,bm, bbm}
\usepackage{paralist}

\usepackage{fourier}
\newcommand{\R}{\mathbb{R}}
\newcommand{\N}{\mathbb{N}}
\newcommand{\C}{\mathbb{C}}
\newcommand{\Z}{ \mathbb{Z}}

\newcommand{\dif}{\mathrm{d}}

\newcommand{\eps}{\varepsilon}
\newcommand{\ii}{\mathrm{i}}
\newcommand{\ee}{\mathrm{e}}
\newcommand*{\open}[1]{#1^\circ}
\newcommand*{\close}[1]{\overline{#1}}

\DeclareMathOperator{\sgn}{sgn}
\DeclareMathOperator{\Pro}{\mathbb{P}}
\DeclareMathOperator{\Exp}{\mathbb{E}}
\DeclareMathOperator{\Uni}{Unif}

\DeclareMathOperator{\Rate}{\mathbb{I}}
\DeclareMathOperator{\Lip}{Lip}

\renewcommand{\Re}{\operatorname{Re}}  %Realteil
\newtheorem{thm}{Theorem}[section]

\newtheorem{cor}[thm]{Corollary}

\newtheorem{proposition}[thm]{Proposition}

\newtheorem{rem}[thm]{Remark}
\newtheorem{lem}[thm]{Lemma}
\newtheorem{prob}[thm]{Problem}

\newtheorem{thmalpha}{Theorem}

\theoremstyle{definition}
\newtheorem{ex}[thm]{Example}

\allowdisplaybreaks
\begin{document}

\title{\bf The large deviation behavior of lacunary sums}
\medskip

\author{Lorenz Fr\"uhwirth, Michael Juhos, Joscha Prochno}

%\thanks{~}

%\keywords{}
%\subjclass{}
%% NB There should be only one primary classification, and zero or
%more secondary classifications.

\date{}

\maketitle

\begin{abstract}
\small
We study the large deviation behavior of lacunary sums $(S_n/n)_{n\in\N}$ with $S_n:= \sum_{k=1}^nf(a_kU)$, $n\in\N$, where $U$ is uniformly distributed on $[0,1]$, $(a_k)_{k\in\N}$ is an Hadamard gap sequence, and $f\colon \R\to\R$ is a $1$-periodic, (Lipschitz-)continuous mapping. In the case of large gaps, we show that the normalized partial sums satisfy a large deviation principle at speed $n$ and with a good rate function which is the same as in the case of independent and identically distributed random variables $U_k$, $k\in\N$, having uniform distribution on $[0,1]$. When the lacunary sequence $(a_k)_{k\in\N}$ is a geometric progression, then we also obtain large deviation principles at speed $n$, but with a good rate function that is different from the independent case, its form depending in a subtle way on the interplay between the function $f$ and the arithmetic properties of the gap sequence. Our work generalizes some results recently obtained by Aistleitner, Gantert, Kabluchko, Prochno, and Ramanan [Large deviation principles for lacunary sums, preprint, 2020] who initiated this line of research for the case of lacunary \emph{trigonometric} sums.
\medspace
\vskip 1mm
\noindent{\bf Keywords}. {Hadamard gap sequence, large deviation principle, large gap condition, geometric progression}\\
{\bf MSC}. Primary 42A55; 60F10; 11L03; Secondary 37A05; 11D45; 11K70
\end{abstract}

%\tableofcontents

% % % % % % % % % % % % % % % % % % % % % % % % % % % % % % % % %

\section{Introduction \& Main results}

The study of lacunary (trigonometric) series is a classical but still flourishing topic in harmonic analysis. Its origins can be traced back to the work of Rademacher \cite{rademacher1922} in 1922, who studied the convergence of series of the form
\begin{equation}
\label{EqRademacherSum}
\sum_{k=1}^{ \infty } b_k r_k( \omega ),
\end{equation}
where $ \omega \in [0,1 ]$, $b = (b_k )_{k \in \N }$, and $r_k$ denotes the $k$\textsuperscript{th} Rademacher function $r_k( \omega ) = \sgn\bigl(\sin ( 2^k \pi \omega ) \bigr)$. Rademacher proved that under the assumption of square summability $\sum_{ k=1}^{ \infty } b_k^2 < \infty $ such a series converges for almost every $ \omega \in [0,1]$. It was then in 1925 that Kolmogorov and Khinchin \cite{KolmogoroffChinchine1925} discovered the necessity of square summability in an even more general setting. Observing that for Rademacher sums \eqref{EqRademacherSum} one has the relation 
\begin{equation*}
\sum_{ k=1}^{ \infty } b_k r_k( \omega )= \sum_{ k=1}^{ \infty } b_k r_1(2^{k-1} \omega ),
\end{equation*}
where in the sum on the right-hand side we have a fixed function $r_1$ containing a sequence of exponentially growing dilation factors, marks the beginning of the study of lacunary trigonometric series of the form
\begin{equation*}
\sum_{ k=1}^{\infty} b_k \cos ( 2 \pi a_k \omega ) \qquad{} \text{and} \qquad{} \sum_{ k=1}^{\infty} b_k \sin ( 2 \pi a_k \omega ),
\end{equation*}
where $\omega \in [0,1]$, $(b_k)_{ k \in \N}$ is a sequence of real numbers, $(a_k)_{ k \in \N }$ is a sequence of positive integers that is lacunary, in the sense that it satisfies the Hadamard gap condition 
\begin{equation*}
\frac{a_{k+1}}{a_k} \geq q > 1 , \qquad{} \text{ for every } k \in \N. 
\end{equation*}
Kolmogorov \cite{Kolmogoroff1924} showed convergence of such series under the $\ell_2$-assumption on $(b_k)_{k \in \N}$ and a few years later Zygmund \cite{zygmund1930} showed that this assumption is necessary. 
\par{}
Coming back to the Rademacher functions $(r_k)_{k \in \N}$, one can easily check that they form a system of independent random variables. Thus, it is natural to ask weather $ \sum_{ k \in \N} b_k r_k$ satisfies a central limit theorem (CLT). Under the assumptions that $(b_k)_{k \in \N} \notin \ell_2$ and $ \max_{ 1 \leq k \leq n } \lvert b_k \rvert = o( \lVert (b_k)_{k \in \N} \rVert_2)$ one can verify that Lindeberg's condition holds and thus we have that for $t \in \R$
\begin{equation*}
\lim_{n \to \infty} \lambda \Bigl( \Bigl\{ \omega \in [0,1] \colon \sum_{ k=1}^n b_k r_k( \omega ) \leq t \lVert (b_k)_{ k=1 }^n \rVert_2 \Bigr\} \Bigr) = \frac{1}{ \sqrt{2 \pi} } \int_{- \infty }^{t} \ee^{-y^2/2} \,\dif y.
\end{equation*} 
In particular this is true when $b_k = 1$ for all $k \in \N$. Kac proved in 1939 a similar CLT for a sequence $(a_k)_{ k \in \N }$ with large gaps, that is, if $ a_{k+1}/a_k \rightarrow \infty $, as $k \rightarrow \infty$ in case of lacunary trigonometric functions. A few years later, in 1947, Salem and Zygmund \cite{SalemZyg1947} showed the CLT under the Hadamard gap condition. More precisely they showed that if $(a_k)_{k \in \N}$ satisfies $a_{k+1}/a_k  \geq q > 1$, we have
\begin{equation*}
\lim_{n \to \infty} \lambda \Bigl( \Bigl\{  \omega \in [0,1] \colon \sum_{ k=1}^n  \cos( 2 \pi a_k \omega) \leq t \sqrt{n/2} \Bigr\}   \Bigr) = \frac{1}{ \sqrt{2 \pi} } \int_{- \infty }^{t} \ee^{-y^2/2} \,\dif y.
\end{equation*}
It became more and more evident that lacunary sums show a behavior similar to sums of truly independent random variables. Salem and Zygmund \cite{SalemZyg1950} in 1950 and Erd\H{o}s and G\'al \cite{ErdoesGal1955} in 1955 were able to prove a law of the iterated logarithm under the Hadamard gap condition, that is, for almost every $ \omega \in [0,1]$ we have
\begin{equation*}
\limsup_{n \rightarrow \infty } \frac{ \sum_{k=1}^n  \cos ( 2 \pi a_k  \omega )}{ \sqrt{ n \log \log n } } = 1.
\end{equation*}
It is therefore natural to ask whether results like these can be generalized to arbitrary periodic functions. This is not always possible as a famous example of Erd\H{o}s and Fortet (see, e.g.,~\cite{ErdoesFortetEx}) showed, where the law of the iterated logarithm failed to be true even for very simple trigonometric polynomials. In 1946, Kac \cite{KacCLT} considered functions $f \colon \R \rightarrow \R $ with
\begin{equation*}
\int_{0}^1 f(\omega ) \,\dif\omega  = 0 \quad{} \text{ and } \quad{} f( \omega +1) = f( \omega)
\end{equation*}
for all $ \omega \in [0,1]$. He was able to show that if one additionally assumes that $f$ is of bounded variation or Hölder-continuous, then $ \frac{1}{\sqrt{n}} \sum_{ k=1}^n f(2^{k - 1} \omega )$ converges in distribution to a normal distribution with mean $0$ and variance
\begin{equation*}
\sigma^2 := \int_{0}^1 f( \omega )^2 \,\dif\omega + 2 \sum_{ k=1}^{ \infty } \int_{0}^{1} f( \omega ) f( 2^k \omega ) \,\dif\omega,
\end{equation*}
provided the latter exists. The form of the variance in the limit is somewhat unexpected, since one would naively assume the same variance as in the independent case, namely $ \int_{0}^{1} f( \omega )^2 \,\dif\omega$. This observation shows that not only the regularity of the function $f$ plays a role, also the arithmetic structure of the sequence $(a_k)_{k \in \N}$ claims its influence. This phenomenom became even more visible when Gapo\v skin \cite{gaposkin} linked the existence of a CLT to the number of solutions of certain Diophantine equations. Only a few years ago Aistleitner and Berkes \cite{AistBerk} improved his result. While both CLT and LIL are quite well understood for lacunary series, fluctuations on the scale of large deviations were not considered till very recently, when in 2020 Aistleitner, Gantert, Kabluchko, Prochno, and Ramanan \cite{agkpr2020} initiated the study of large deviation principles (LDP) for lacunary \emph{trigonometric} sums, obtaining a series of unexpected results that display a subtle dependence on the arithmetic structure of the gap sequence that is not visible in the trigonometric setting on the scales of a CLT and a LIL; for the latter two the arithmetic structure is irrelevant and they always display a behavior as in the case of independent and identically distributed random variables.

In this paper we continue the study of large deviations for lacunary series and make progress on some of the open problems stated in \cite{agkpr2020}. More precisely, we will consider Hadamard gap sequences $(a_k)_{k \in \N}$ and study the tail behavior of partial sums of the form
\begin{equation}
\label{IntroEqLacSum}
\sum_{ k=1}^n f( a_k \omega ),
\end{equation} 
where $\omega \in [0,1]$ and $f \colon \R \rightarrow \R $ is a general $1$-periodic function satisfying certain regularity assumptions. Functions like the one in \eqref{IntroEqLacSum} can be interpreted as random variables on $[0,1]$ equipped with the Borel-sigma field and the Lebesgue measure $\lambda$ on it. 
%In Section~\ref{SectionResults} we state our results, the proofs are given in Section~\ref{SectionProofs}. 
\par{}

\subsection{Main Results}
\label{SectionResults}

Before presenting the main results, recall that a sequence $(X_n)_{n \in \N}$ of random variables satisfies an LDP at speed $(s_n)_{n \in \N}$ and rate function $\Rate \colon \R \rightarrow [0, \infty]$ if, for every Borel measurable set $A \subset \R$, we have 
\begin{equation*}
- \inf_{ x \in \open{A}} \Rate(x) \leq \liminf_{n \rightarrow \infty } \frac{1}{s_n} \log\Pro[X_n \in A] \leq \limsup_{n \rightarrow \infty } \frac{1}{s_n} \log\Pro[X_n \in A]
\leq - \inf_{ x \in \close{A}} \Rate(x),
\end{equation*}
where $\open{A}$ and $\close{A} $ denote the interior and the closure of $A$, respectively. The speed $(s_n)_{n \in \N}$ is a sequence of positive real numbers tending to infinity and the rate function $\Rate$ is lower-semicontinuous. If $\Rate$ has compact level sets, we speak of a good rate function (GRF). The classical setting of independent and identically distributed (i.i.d.)\ random variables is dealt with in Cram\'er's theorem \cite{CramerThm1938}: for a sequence of i.i.d.\ random variables $(X_n)_{n \in \N}$ with finite exponential moments, i.e.,
\begin{equation*}
\Lambda(u):= \Exp\bigl[ \ee^{ u X_1} \bigr] < \infty
\end{equation*} 
for all $u$ in a neighborhood of $0$, one has that
\begin{equation*}
\lim_{n \to \infty} \frac{1}{n} \log\Pro[ X_1 + \dotsb + X_n \geq n t] = - \Lambda^{*}(t),
\end{equation*}
where $t > \Exp[X_1]$ and $ \Lambda^* \colon \R \rightarrow [0, \infty ]$ is the Legendre--Fenchel transform of $\Lambda$.

Now, let $U \sim \Uni([0,1])$ be a random variable with the uniform distribution on $[0,1]$. Given a sequence of positive integers $(a_k)_{k \in \N}$ and $f \colon \R \rightarrow \R$, a measurable $1$-periodic function, we define
\begin{equation}
\label{MainResEqLacSum}
S_n := \sum_{k=1}^{n} f(a_k U), \qquad n\in\N.
\end{equation} 
Our aim is to prove LDPs for the sequence $(S_n/n)_{n \in \N}$ in two different and natural settings, namely for large gaps and for gap sequences forming a geometric progression.

% % % % % % % % % % % % % % % % % % % % %
\subsubsection{The case of large gaps}
% % % % % % % % % % % % % % % % % % % % %

We first consider the case of independent and identically distributed random variables. Assume that for the function $f \colon \R \rightarrow \R$,
\begin{equation}
\label{MainResEqiidCumGenFct}
\widetilde{ \Lambda}^f( \theta ) := \log \int_{0}^{1} \ee^{ \theta f( \omega ) } \,\dif\omega
\end{equation}
exists for all $ \theta \in \R$. Let $(U_k )_{ k \in \N} $ be an i.i.d.\ sequence  of random variables with the same distribution as $U$ and define
\begin{equation}
\label{MainResEqiidSum}
\widetilde{S}_n := \sum_{k=1}^{n} f(a_k U_k),\qquad n\in\N.
\end{equation}
By Cramér's theorem, $( \widetilde{S}_n/n)_{n \in \N}$ satisfies an LDP in $\R$ at speed $n$ and with GRF $\widetilde{I}^f \colon \R \rightarrow [0, \infty ]$ given by the Legendre--Fenchel transform of $\widetilde{\Lambda}^f $, that is,
\begin{equation}
\label{MainResEqiidGRF}
\widetilde{I}^f(x) := \sup_{ \theta \in \R} \bigl[ \theta x - \widetilde{ \Lambda}^f(\theta) \bigr].
\end{equation} 
Our first result treats the case where the Hadamard gap sequence $(a_k)_{k \in \N}$ has ``large gaps'', i.e.,
\begin{equation*}
\frac{a_{k+1}}{a_k} \stackrel{k \rightarrow \infty}{\longrightarrow} \infty .
\end{equation*}

\begin{thmalpha}
\label{MainResThmLargeGaps}
Let $f \colon \R \to \R $ be a $1$-periodic continuous function, let \(U \sim \Uni([0, 1])\) and \((a_k)_{k \in \N }\) be a lacunary sequence with large gaps. Then the sequence $( S_n/n)_{n \in \N }$ satisfies an LDP in $\R $ at speed $n$ and with GRF $\widetilde{I}^f$, where \(S_n\) is defined as in Equation~\eqref{MainResEqLacSum}.
\end{thmalpha}

\begin{rem}
In essence, modulo several technicalities, the proof (given in Section~\ref{SectionProofs}) uses uniform approximation of continuous functions via trigonometric polynomials and the G\"artner--Ellis theorem. We also provide a proof in a more restrictive setting using the Fourier expansion technique (see Proposition~\ref{PropFourierMeth}), where we put a growth condition on the Fourier coefficients of the function \(f\). We do this with a view towards potential generalizations, e.g., for certain non-continuous $f$.
\end{rem}

% % % % % % % % % % % % % % % % % % % % % % % % % % %
\subsubsection{The case of geometric progressions}
% % % % % % % % % % % % % % % % % % % % % % % % % % %

We now consider a lacunary sequence $(a_k)_{k \in \N}$ of the form $a_k=q^k$ for some $q \in \N$ with $q \geq 2$. In this case, the large deviation behavior changes dramatically as is shown in the next theorem.

\begin{thmalpha}
\label{MainResThmGeoProg}
Let $f \colon \R \rightarrow \R $ be $1$-periodic and Lipschitz continuous. Let $( a_k)_{k \in \N}$ be a geometric sequence, i.e., $ a_k= q^k$ for some $q \in \N$ with $q \geq 2$, and let $(S_n/n)_{n \in \N}$ be as defined in Equation \eqref{MainResEqLacSum}. Then
\begin{equation*}
\Lambda_q^f( \theta) := \lim_{n \to \infty} \frac{1}{n} \log\Exp\bigl[ \ee^{ \theta S_n} \bigr]
\end{equation*}
exists with the convergence holding uniformly on compact subsets of an open set $\mathcal{D}$ in the complex plane such that $\R \subset \mathcal{D}$. Moreover, $(S_n/n)_{n \in \N }$ satisfies an LDP in $\R $ at speed $n$. As GRF we obtain the Legendre--Fenchel transform $I_q^f \colon \R \rightarrow [0, \infty]$ of $  \Lambda_{q}^f$, i.e.,
\begin{equation}
\label{MainResEqGRF}
I_q^f(x) := \sup_{ \theta \in \R} \bigl[ \theta x - \Lambda_q^f(\theta) \bigr].
\end{equation} 
Furthermore, the limit $\lim_{ q \rightarrow \infty} I_q^f(x) = \widetilde{I}^f(x)$ holds uniformly on compact subsets of the interval $(-1,1)$.
\end{thmalpha}

We are even able to strengthen Theorem~\ref{MainResThmGeoProg} to arbitrary continuous functions.

\begin{cor}
\label{MainResKorGeoCase}
Let $( a_k)_{k \in \N}$ be a geometric sequence, i.e., $ a_k= q^k$ for some $q \in \N$ with $q \geq 2$, and let $f \colon \R \rightarrow \R$ be a continuous $1$-periodic function. Then, $(S_n/n)_{n \in \N}$ from Equation \eqref{MainResEqLacSum} satisfies an LDP in $\R$ at speed $n$ and some GRF $\Rate_q^f \colon \R \rightarrow [0, \infty]$.
\end{cor}

Before we present the necessary large deviation background and then the proofs of the main results, we close this section providing some instructive examples.

\begin{ex}
\begin{asparaenum}
\item \(f(\omega) = \cos(2 \pi \omega)\) and \(a_k = 2^k\). Here the asymptotic cumulant generating function is \(\Lambda_2^f = \Lambda_2\) as given in Theorem~B in~\cite{agkpr2020}. There is no closed expression for it, but the coefficients of its Taylor series may be computed explicitly; see Lemma~3.5 and Appendix~A \textit{loc.~cit.}\ for more information.

\item \(f(\omega) = \cos(2 \pi \omega) + \cos(4 \pi \omega)\) and \(a_k = 2^k\). We can rewrite \(S_n\) in the following way,
\begin{align*}
S_n(\omega) &= \sum_{k = 1}^n \bigl( \cos(2^{k + 1} \pi \omega) + \cos(2^{k + 2} \pi \omega) \bigr)\\
&= 2 \Big(\sum_{k = 1}^n \cos(2^{k + 1} \pi \omega)\Big) + \cos(2^{n + 2} \pi \omega) - \cos(4 \pi \omega)\\
&=: 2 S_n^{(1)}(\omega) + r(\omega),
\end{align*}
where \(S_n^{(1)} = \sum_{k = 1}^n \cos(2^{k + 1} \pi \omega)\) is the lacunary sum implicitly dealt with in Example~1, and \(\lvert r(\omega) \rvert = \lvert \cos(2^{n + 2} \pi \omega) - \cos(4 \pi \omega) \rvert \leq 2\). Therefore,
\begin{equation*}
\frac{1}{n} \log\bigl( \Exp\bigl[ \ee^{\theta S_n} \bigr] \bigr) \leq \frac{2 \theta}{n} + \frac{1}{n} \log\bigl( \Exp\bigl[ \ee^{2 \theta S_n^{(1)}} \bigr] \bigr) \stackrel{n \to \infty}{\longrightarrow} \Lambda_2(2 \theta),
\end{equation*}
and analogously,
\begin{equation*}
\frac{1}{n} \log\bigl( \Exp\bigl[ \ee^{\theta S_n} \bigr] \bigr) \geq -\frac{2 \theta}{n} + \frac{1}{n} \log\bigl( \Exp\bigl[ \ee^{2 \theta S_n^{(1)}} \bigr] \bigr) \stackrel{n \to \infty}{\longrightarrow} \Lambda_2(2 \theta),
\end{equation*}
and this implies \(\Lambda_2^f(\theta) = \Lambda_2(2 \theta)\). This is clearly distinct from \(\Lambda_2\), but note also that we could have obtained the same \(\Lambda_2^f\) if we had started with \(g(\omega) := 2 \cos(2 \pi \omega)\). Thus, we may take away that, for identical \(q\), different \(f\) can result in different \(\Lambda_q^f\), though there is no bijection.

\item \(f(\omega) = \cos(2 \pi \omega) - \cos(4 \pi \omega)\) and \(a_k = 2^k\). Here \(S_n\) will telescope,
\begin{align*}
S_n &= \sum_{k = 1}^n \bigl( \cos(2^{k + 1} \pi \omega) - \cos(2^{k + 2} \pi \omega) \bigr)\\
&= \cos(4 \pi \omega) - \cos(2^{n + 2} \pi \omega),
\end{align*}
this means \(\lvert S_n \rvert \leq 2\) for all \(n \in \N\) and hence
\begin{equation*}
\Lambda_2^f(\theta) = \lim_{n \to \infty} \frac{1}{n} \log\bigl( \Exp\bigl[ \ee^{\theta S_n} \bigr] \bigr) = 0.
\end{equation*}
Therefore, we get a markedly different rate function in this case (the trivial one, to be precise), although \(f\) is nontrivial and seems to differ from the previous example only marginally.

\item \(f(\omega) = \cos(2 \pi \omega) + \cos(4 \pi \omega)\) and \(a_k = 2^k + 1\). Now \((a_k)_{k \in \N}\) is no longer a geometric sequence, yet because of 1-periodicity we have
\begin{align*}
S_n(\omega) &= \sum_{k = 1}^n \bigl( \cos(2 (2^k + 1) \pi x) + \cos(2 (2^{k + 1} + 1) \pi \omega) \bigr)\\
&= \sum_{k = 1}^n \bigl( \cos(2^{k + 1} \pi x) + \cos(2^{k + 2} \pi \omega) \bigr),
\end{align*}
which is exactly the same as in Example~2.
\end{asparaenum}
\end{ex}

Together these examples demonstrate that the rate function, and hence the whole LDP, does depend on the interplay between the function \(f\) and the sequence \((a_k)_{k \in \N}\), not solely on the one or the other. In particular it is not only the arithmetic properties of \((a_k)_{k \in \N}\) that influence the LDP: in the last example it is neither a geometric sequence nor does it have large gaps, so it does not fall into the framework of the present Theorems~\ref{MainResThmLargeGaps} and~\ref{MainResThmGeoProg}, and still an LDP can be proved.

% % % % % % % % % % % % % % % % % % % % % % % % % % % % % % % % %
\section{Elements of large deviations theory -- some elementary background}
% % % % % % % % % % % % % % % % % % % % % % % % % % % % % % % % %

Large deviations describe the decay of the probability of rare events on an exponential scale. In contrast to the law of large numbers and the CLT, large deviations behavior is highly non-universal. In 1938 Cramér \cite{CramerThm1938} published his famous theorem (see also Theorem~2.2.3 in \cite{dezei1998}): for a sequence of independent and identically distributed random variables $(X_n)_{n \in \N}$ with finite exponential moments, i.e.,
\begin{equation*}
\Lambda(u):= \Exp\bigl[ \ee^{ u X_1} \bigr] < \infty
\end{equation*} 
for all $u$ in a neighborhood of $0$, one has that
\begin{equation*}
\lim_{n \to \infty} \frac{1}{n} \log\Pro[ X_1 + \dotsb + X_n \geq n t] = - \Lambda^{*}(t),
\end{equation*}
where $t > \Exp[X_1]$. The function $ \Lambda^* \colon \R \rightarrow [0, \infty ]$ is the Legendre--Fenchel transform of $\Lambda$, i.e.,
\begin{equation*}
\Lambda^*(x) := \sup_{ t \in \R}[xt - \Lambda(t)] .
\end{equation*}
A few decades later, Donsker and Varadhan initiated the systematic study of large deviations and generalized Cramér's idea (see \cite{dezei1998} for more historical background). A sequence $(X_n)_{n \in \N}$ of random variables (not necessarily i.i.d.)\ satisfies an LDP at speed $(s_n)_{n \in \N}$ with rate function $\Rate \colon \R \rightarrow [0, \infty]$, if for every Borel-measurable set $A \subset \R$ we have 
\begin{equation*}
- \inf_{ x \in \open{A}} \Rate(x) \leq \liminf_{n \rightarrow \infty } \frac{1}{s_n} \log\Pro[X_n \in A] \leq \limsup_{n \rightarrow \infty } \frac{1}{s_n} \log\Pro[X_n \in A]
\leq - \inf_{ x \in \close{A}} \Rate(x),
\end{equation*}
where $\open{A}$ and $\close{A} $ denote the interior and the closure of $A$, respectively. The speed $(s_n)_{n \in \N}$ is a sequence of positive real numbers converging to infinity and the rate function $\Rate$ is lower-semicontinuous. If $\Rate$ has compact level sets, we speak of a good rate function (GRF). An important result from large deviations theory is the G\"artner--Ellis theorem, which we shall use several times in this paper. Having its roots in a paper of G\"artner from 1977 \cite{Gaertner77}, Ellis established this result in 1984 \cite{Ellis1984}. The reader may also consult \cite[Theorem~2.3.6]{dezei1998}.

\begin{thm}
\label{ThmGaertnerEllis}
Let $(Z_n)_{ n \in \N}$ be a sequence of real valued random variables and assume that the following limit exists and is differentiable on $\R$:
\begin{equation}
\label{EqGELim}
\Lambda( \theta ) := \lim_{n \rightarrow \infty } \frac{1}{n} \log \Exp\bigl[ e^{n \theta Z_n}  \bigr] .
\end{equation}
Then, $(Z_n)_{n \in \N}$ satisfies an LDP at speed $n$ and with GRF $\Lambda^{*}\colon \R \rightarrow [0, \infty ]$, where
\begin{equation*}
\Lambda^{*}(x) := \sup_{ \theta \in \R } \bigl[ x \theta - \Lambda( \theta )  \bigr]
\end{equation*}
is the Legendre--Fenchel transform of $\Lambda$. 
\end{thm}

% % % % % % % % % % % % % % % % % % % % % %
\section{Proofs}
\label{SectionProofs}
% % % % % % % % % % % % % % % % % % % % % % 

In this section we provide the proofs of all statements in Section~\ref{SectionResults}. The central idea is to establish LDPs for trigonometric polynomials and use uniform approximation of continuous functions by such polynomials. We start with a lemma that will later facilitate the application of the G\"artner--Ellis theorem. In the following, if, e.g., $j$ is an integer index, $ j \in [a,b]$ means $ j \in \{ k \in \Z \colon a \leq k \leq b \} $, where $ a,b \in \Z$.

\begin{lem}\label{lem:trigon_poly}
Let \(d \in \N\), define \(p_d(z) := \sum_{j = 0}^d \frac{1}{j!} z^d \) for \( z \in \C \) , let \(m \in \N_0\) and define \(f \colon \R \to \C\) by \(f(x) := \sum_{j = -m}^m c_j \ee^{2 \pi \ii j x}\)  with \(c_{-m}, \dotsc, c_m \in \C\). Then, for all \(\theta, x \in \R\), we have
\begin{equation*}
p_d(\theta f(x)) = \sum_{j = -d m}^{d m} b_j(\theta, d)  \ee^{2 \pi \ii j x},
\end{equation*}
where the coefficients \(b_j(\theta, d)\), \(j \in [-d m, d m]  \), are complex numbers depending on \(c_{-m},\dotsc, c_m\) and with
\begin{equation}\label{eq:limit b_0(theta,d)}
b_0(\theta, d) = \int_0^1 p_d(\theta f(\omega)) \, \dif\omega \stackrel{d\to\infty}{\longrightarrow} \int_0^1 \ee^{\theta f(\omega)} \, \dif\omega.
\end{equation} 
Furthermore, given a lacunary sequence \((a_k)_{k \in \N }\) with large gaps, there exists \(k_0 = k_0(d, m) \in \N\) such that for every \(\theta \in \R\) and for all \(n > k_0\),
\begin{equation*}
\int_0^1 \prod_{k = k_0 + 1}^n p_d(\theta f(a_k \omega)) \, \dif\omega = b_0(\theta, d)^{n - k_0}.
\end{equation*}
\end{lem}

\begin{proof}
A proof by induction reveals that, for all \(j \in \N_0\),
\begin{equation*}
f(\omega)^j = \sum_{k = -j m}^{j m} c_k^{(j)} \ee^{2 \pi \ii k \omega},
\end{equation*}
with some numbers \(c_k^{(j)}\in \C\), \(k \in [-j m, j m]\), the precise knowledge of which is irrelevant. This yields
\begin{align*}
p_d(\theta f(\omega)) &= \sum_{j = 0}^d \frac{\theta^j}{j!} \sum_{k = -j m}^{j m} c_k^{(j)} \ee^{2 \pi \ii k \omega}\\
&= \sum_{k = -d m}^{d m} \biggl( \sum_{j = \lceil \lvert k \rvert / m \rceil}^d \frac{c_k^{(j)} \, \theta^j}{j!} \biggr) \ee^{2 \pi \ii k \omega}\\
&= \sum_{j = -d m}^{d m} b_j(\theta, d) \ee^{2 \pi \ii j \omega},
\end{align*}
where we put
\begin{equation*}
b_j(\theta, d) := \sum_{k = \lceil \lvert j \rvert / m \rceil}^d \frac{c_j^{(k)} \, \theta^k}{k!}.
\end{equation*}
A simple computation reveals that this implies the integral representation of \(b_0(\theta, d)\) in the statement of the lemma, i.e.,
\begin{equation*}
b_0(\theta, d) = \int_0^1 p_d(\theta f(\omega)) \, \dif\omega.
\end{equation*}
Because \(p_d\) approximates the exponential function uniformly on bounded subsets of $\C$ (in particular on compact subsets) and since $\theta f([0,1])\subset \C$ is compact,
%(see also Lemma~3.2 in \cite{agkpr2020}),
integration and limit for \(d \to \infty\) may be exchanged, which shows that
\begin{equation}
\lim_{d\to\infty}b_0(\theta, d) = \int_0^1 \ee^{\theta f(\omega)} \, \dif\omega.
\end{equation}

Now let \(\theta \in \R\) and \(n > k_0\), where \(k_0 \in \N\) is such that \(\frac{a_{k + 1}}{a_k} \geq m d + 1\) for all \(k \geq k_0\). Then we have, for all \(\omega \in [0, 1]\),
\begin{align*}
\prod_{k = k_0 + 1}^n p_d(\theta f(a_k \omega)) &= \prod_{k = k_0 + 1}^n \, \sum_{j = -d m}^{d m} b_j(\theta, d) \ee^{2 \pi \ii j a_k \omega}\\
&= \sum_{j_{k_0 + 1}, \dotsc, j_n = -d m}^{d m}\, \prod_{k = k_0 + 1}^n \bigl( b_{j_k}(\theta, d) \ee^{2 \pi \ii j_k a_k \omega} \bigr)\\
&= \sum_{j_{k_0 + 1}, \dotsc, j_n = -d m}^{d m}\, \biggl( \prod_{k = k_0 + 1}^n b_{j_k}(\theta, d) \biggr) \exp\biggl( 2 \pi \ii \omega \sum_{k = k_0 + 1}^n j_k a_k \biggr).
\end{align*}
It is left to show that if \(j_k \neq 0\) for some \(k \in [k_0 + 1, n]\), then \(\sum_{k = k_0 + 1}^n j_k a_k \neq 0\), because in that case the integral of the exponential function evaluates to zero and the only nontrivial summand left is the one with \(j_{k_0 + 1} = \dotsb = j_n = 0\), which then yields the desired \(b_0(\theta, d)^{n - k_0}\). So let \((j_{k_0 + 1}, \dotsc, j_n) \neq 0\) and let \(\ell := \max\{k \in [k_0 + 1, n] \colon j_k \neq 0\}\). By the choice of \(k_0\), \(a_{\ell} \geq a_k (d m + 1)^{\ell - k}\) for all \(k \in [k_0 + 1, \ell]\). We now consider two cases according to the sign of \(j_{\ell}\). First, suppose that \(j_{\ell} > 0\). Then, because \(j_{\ell} \geq 1\) and \(j_k \geq -dm\) for all \(k \in [k_0 + 1, \ell - 1]\), we obtain
\begin{equation*}
\sum_{k = k_0 + 1}^n j_k a_k = \sum_{k = k_0 + 1}^\ell j_k a_k  = j_{\ell} a_{\ell} + \sum_{k = k_0 + 1}^{\ell-1} j_k a_k \geq a_{\ell} - d m \sum_{k = k_0 + 1}^{\ell - 1} a_k.
\end{equation*}
Using \(a_k \leq a_\ell (d m + 1)^{k - \ell}\) for all \(k \in [k_0 + 1, \ell]\), this can be estimated further and we obtain
\begin{align*}
a_{\ell} - d m \sum_{k = k_0 + 1}^{\ell - 1} a_k & \geq a_{\ell} \biggl( 1 - d m \sum_{k = k_0 + 1}^{\ell - 1} (d m + 1)^{k - \ell} \biggr)\\
& = a_\ell\Biggl( 1 - d m \, (d m + 1)^{k_0 - \ell}\sum_{k=1}^{\ell-k_0-1} (dm+1)^k\Biggr)\\ 
&= a_{\ell} \Biggl( 1 - d m \, \frac{1 - (d m + 1)^{k_0 + 1 - \ell}}{d m} \Biggr)\\
& = a_{\ell} (d m + 1)^{k_0 + 1 - \ell} > 0.
\end{align*}
The case \(j_{\ell} < 0\) is argued analogously. This completes the proof.
\end{proof}

We use the previous lemma to prove Theorem~\ref{MainResThmLargeGaps} by employing the Stone--Weierstraß theorem (see, e.g., \cite[Chapter~15]{Koenigsberger2004}).

\begin{proof}[Proof of Theorem~\ref{MainResThmLargeGaps}]
The proof is done in two steps: first we assume that \(f\) is a trigonometric polynomial and then we use approximation for arbitrary \(f\).

\textit{Step~1:} The general idea is the same as in the proof of Theorem~A in \cite{agkpr2020}; in particular, \cite[Lemma~3.2]{agkpr2020} can be adapted to hold for any bounded function \(f\) instead of \(\omega \mapsto \cos(2 \pi a_k \omega)\).
Now let \(f \colon \omega \mapsto \sum_{j = -m}^m c_j \ee^{2 \pi \ii j \omega}\) be a trigonometric polynomial for some \(m \in \N_0\) and \(c_{-m}, \dotsc, c_m \in \C\) (actually \(c_{-j} = \overline{c_j}\) holds for all \(j \in [0, m]\) because of \(f(\R) \subset \R\)).
For fixed $\theta \in \R$ and $\eps > 0$, we can find some $d_0 = d_0(\varepsilon) \in \N$ such that \(\lim_{\varepsilon \to 0} d_0(\varepsilon) = \infty\) and for all $d \geq d_0$ and all $ x \in [ - \theta \lVert f \rVert_{ \infty } , \theta \lVert f \rVert_{ \infty }]$ (note that $\lVert f \rVert_{ \infty}$ is finite for a trigonometric polynomial $f$) we have
\begin{equation}
\label{EqEstExpviaPoly}
1 - \eps \leq \frac{\ee^x}{p_d(x)} \leq 1 + \eps,
\end{equation}
because $p_d$, as defined in Lemma~\ref{lem:trigon_poly}, is the $d$\textsuperscript{th} Taylor polynomial of the exponential function. Recall that, by Lemma~\ref{lem:trigon_poly}, for fixed $d \geq d_0$ we can find $k_0 \in \N$ such that, for all \(\theta \in \R\) and \(n > k_0\),
\begin{equation}
\label{EqIntProdpd}
\int_0^1 \prod_{k = k_0 + 1}^n p_d(\theta f(a_k \omega)) \, \dif\omega = b_0(\theta, d)^{n - k_0}.
\end{equation}

Combining ~\eqref{EqEstExpviaPoly} and ~\eqref{EqIntProdpd} results in the estimate 
\begin{equation*}
\Exp\bigl[ \ee^{ \theta  S_n}   \bigr] = \int_{0}^{1} \prod_{k=1}^{n} \ee^{ \theta f( a_k \omega ) } \,\dif\omega  \leq C_1 (1 + \eps )^{ n-k_0} \int_{0}^1 \prod_{k=k_0 +1}^{n} p_d( \theta f( a_k \omega) ) \,\dif\omega = C_1 (1 + \eps )^{ n-k_0} b_0(\theta, d)^{n - k_0},
\end{equation*}
where we have estimated $ \prod_{k=1}^{k_0} \ee^{ \theta f( a_k \omega ) } \leq  \ee^{ k_0 \theta \lVert f \rVert_{ \infty }} =: C_1 $ for $\omega \in [0,1]$. Analogously one receives the lower bound, where we have
\begin{equation*}
\Exp\bigl[ \ee^{ \theta  S_n}   \bigr] \geq C_2(1 - \eps )^{ n-k_0} b_0(\theta, d)^{n - k_0}
\end{equation*}
with $C_2 := \ee^{- k_0 \theta \lVert f \rVert_{ \infty }} $.
Together this leads to
\begin{equation*}
\log (1 - \eps ) + \log b_0( \theta , d) \leq  \liminf_{ n \rightarrow \infty } \frac{1}{n} \log \Exp\bigl[ \ee^{ \theta  S_n}   \bigr]
\end{equation*}
and
\begin{equation*}
\limsup_{ n \rightarrow \infty } \frac{1}{n} \log \Exp\bigl[ \ee^{ \theta  S_n}   \bigr]   \leq \log (1 + \eps ) + \log  b_0( \theta , d).
\end{equation*}
Now, if we let $ \eps \rightarrow 0$ (and thus $d \rightarrow \infty$) and combine the previous estimates with \eqref{eq:limit b_0(theta,d)} of Lemma~\ref{lem:trigon_poly}, we obtain
\begin{equation*}
\lim_{n \rightarrow \infty } \frac{1}{n} \log  \Exp\bigl[ \ee^{ \theta  S_n}   \bigr]  = \log   \int_{0}^{1} \ee^{ \theta f(\omega )} \,\dif\omega  = \widetilde{\Lambda}^f( \theta ).
\end{equation*}
Hence, the G\"artner--Ellis limit exists for every \(\theta \in \R\), is finite, and the map \(\widetilde{\Lambda}^f\) is clearly differentiable on \(\R\). Therefore, the claim follows.

\textit{Step~2:} Let \(f\) be an arbitrary continuous and $1$-periodic function. Consider the complex algebra \(D := \bigl\{ \omega \mapsto \sum_{j = -m}^m c_j \ee^{2 \pi \ii j \omega} \colon m \in \N_0,\, c_{-m}, \dotsc, c_{m} \in \C \bigr\}\) of all trigonometric polynomials with period~1. We have \(1 \in D\), \(D\) separates points in \([0, 1)\) (because if \(\ee^{2 \pi \ii \omega} = \ee^{2 \pi \ii \omega'}\), then \(\omega  - \omega' \in \Z\)) and is closed under complex conjugation. By the Stone--Weierstraß theorem there exists a sequence \((f_m)_{m \in \N }\) in \(D\) such that, for each \(m \in \N\),
\begin{equation}\label{eq:uniform approximation f fm}
\lVert f - f_m \rVert_\infty \leq \frac{1}{m}.
\end{equation}
Note that each \(f_m\) can be chosen to be real-valued. Indeed, for any \(\omega \in \R\),
\begin{equation*}
\lvert f(\omega) - \Re f_m(\omega) \rvert = \lvert \Re(f(\omega) - f_m(\omega)) \rvert \leq \lvert f(\omega) - f_m(\omega) \rvert \leq \lVert f - f_m \rVert_\infty
\end{equation*}
and hence
\begin{equation*}
\lVert f - \Re f_m \rVert_\infty \leq \lVert f - f_m \rVert_\infty \leq \frac{1}{m}.
\end{equation*}
We are going to prove that, for any \(\theta \in \R\),
\begin{equation*}
\lim_{n \to \infty} \frac{1}{n} \log\Exp[\ee^{\theta S_n}]  = \log \int_0^1 \ee^{\theta f(\omega)} \, \dif\omega  \mathrel{\Big(=}  \widetilde{\Lambda}^f(\theta)\Big).
\end{equation*}
Because obviously \(\widetilde{\Lambda}^f(\R) \subset \R\) and \(\widetilde{\Lambda}^f\) is differentiable on \(\R\), we can then apply the G\"arnter--Ellis theorem to obtain the claimed LDP with the corresponding rate function. So let \(S_{n}^m := \sum_{k = 1}^n f_m(a_k U)\), $n\in\N$, \(\theta \in \R\), and assume \(\varepsilon > 0\). Then there exists \(m_1 \in \N\) such that, for all \(m \geq m_1\),
\begin{equation*}
\frac{\lvert \theta \rvert}{m} < \frac{\varepsilon}{3}.
\end{equation*}
From the uniform approximation \eqref{eq:uniform approximation f fm} we infer, for any \(n \in \N\) and each realization,
\begin{equation*}
\lvert S_n - S_{n}^m \rvert = \biggl\lvert \sum_{k = 1}^n \bigl( f(a_k U) - f_m(a_k U) \bigr) \biggr\rvert \leq \frac{n}{m}.
\end{equation*}
This in its turn implies
\begin{align*}
\frac{1}{n} \log\Exp\Big[\ee^{\theta S_n}\Big] &\leq \frac{1}{n} \log\Exp\Big[\ee^{\theta S_{n}^m + \lvert \theta \rvert n/m}\Big] = \frac{1}{n} \log \Exp\Big[\ee^{\theta S_{n}^m}\Big] + \frac{\lvert \theta \rvert}{m}\\
& < \frac{1}{n} \log\Exp\Big[\ee^{\theta S_{n}^m}\Big] + \frac{\varepsilon}{3}\\
\intertext{and analogously}
\frac{1}{n} \log\Exp\Big[\ee^{\theta S_n}\Big] &> \frac{1}{n} \log\Exp\Big[\ee^{\theta S_{n}^m}\Big] - \frac{\varepsilon}{3}.
\end{align*}
Furthermore, note that
\begin{equation*}
\lim_{m \to \infty} \int_0^1 \ee^{\theta f_m(\omega)} \, \dif\omega = \int_0^1 \ee^{\theta f(\omega)} \, \dif\omega.
\end{equation*}
This is so, because we have $(f_m)_{m \in \N} \rightarrow f$ uniformly, the exponential function is uniformly continuous on, say, \([-\lvert \theta \rvert \lVert f \rVert_\infty - 1, \lvert \theta \rvert \lVert f \rVert_\infty + 1]\) and therefore also \((\ee^{\theta f_m})_{m \in \N } \to \ee^{\theta f}\) uniformly, which yields convergence of the integrals. This shows the existence of \(m_2 \in \N\) such that, for all \(m \geq m_2\),
\begin{equation*}
\biggl\lvert \log\int_0^1 \ee^{\theta f_m(\omega)} \, \dif\omega - \log \int_0^1 \ee^{\theta f(\omega)} \, \dif\omega \biggr\rvert < \frac{\varepsilon}{3}.
\end{equation*}
Let \(m := \max\{m_1, m_2\}\). Then we know from Step~1 that \(\lim_{n \to \infty} \frac{1}{n} \log\Exp[\ee^{\theta S_{n}^m}] = \log \int_0^1 \ee^{\theta f_m(\omega)} \, \dif\omega\). Thus, there exists \(n_0 \geq 1\) such that, for all \(n \geq n_0\),
\begin{equation*}
\biggl\lvert \frac{1}{n} \log\Exp\Big[\ee^{\theta S_n^m}\Big] - \log \int_0^1 \ee^{\theta f_m(\theta)} \, \dif\omega  \biggr\rvert < \frac{\varepsilon}{3}.
\end{equation*}
Now, for every \(n \geq n_0\), we obtain
\begin{align*}
\biggl\lvert \frac{1}{n} \log\Exp\Big[\ee^{\theta S_n}\Big] - \log \int_0^1 \ee^{\theta f(\omega)} \, \dif\omega  \biggr\rvert &\leq \Bigl\lvert \frac{1}{n} \log\Exp\Big[\ee^{\theta S_n}\Big] - \frac{1}{n} \log\Exp\Big[\ee^{\theta S_{n}^m}\Big] \Bigr\rvert\\
&\quad + \biggl\lvert \frac{1}{n} \log\Exp\Big[\ee^{\theta S_{n}^m}\Big] - \log \int_0^1 \ee^{\theta f_m(\omega)} \, \dif\omega \biggr\rvert\\
&\quad + \biggl\lvert \log \int_0^1 \ee^{\theta f_m(\omega)} \, \dif\omega  - \log \int_0^1 \ee^{\theta f(\omega)} \, \dif\omega \biggr\rvert\\
&< \frac{\varepsilon}{3} + \frac{\varepsilon}{3} + \frac{\varepsilon}{3} = \varepsilon
\end{align*}
and the proof is complete.
\end{proof}

In view of the classical work \cite{KacCLT} of Kac, another possible approach uses Fourier analytic methods. Taking this route allows us to establish a large deviation principle under certain growth conditions on the Fourier coefficients. Although this result is obtained under slightly stronger assumptions, we present it here with a view towards potential improvements in the future, where such an approach might be useful.

\begin{proposition}
\label{PropFourierMeth}
Let $f \colon \R \rightarrow \R$ be a $1$-periodic function with $ \int_{0}^{1} f(x)^2 \,\dif x < \infty $ and $(a_k)_{ k \in \N}$ be a sequence with large gaps. For the Fourier expansion
\begin{equation}
\label{EqFourierExpLem}
f(x) = \sum_{k=- \infty}^{ \infty } c_k \ee^{2 \pi \ii k x},
\end{equation}
we assume that $ \lvert c_k \rvert \leq M \lvert k \rvert^{ - \beta }$ for some $ \beta > 1$ and some constant $M\in(0,\infty)$. Then, $(S_n/n)_{n \in \N}$ satisfies the LDP at speed $n$ with GRF $\widetilde{I}^f$ from \eqref{MainResEqiidGRF}. 
\end{proposition}

\begin{rem}
Note that the assumptions of Proposition~\ref{PropFourierMeth} imply continuity of the function $f$. To see this, we consider a sequence $(x_n)_{n \in \N}$ of real numbers converging to some $ x \in \R$. Then, for fixed $N \in \N$, we have
\begin{align*}
\lvert f(x_n) - f(x) \rvert & = \Bigl\lvert  \sum_{k = - \infty }^{ \infty } c _k \ee^{ 2 \pi \ii k x_n} - \sum_{k = - \infty }^{ \infty } c _k \ee^{ 2 \pi \ii k x}  \Bigr\rvert\\
& \leq 
 \sum_{k = - N }^{ N }  \lvert c _k \rvert \bigl\lvert \ee^{ 2 \pi \ii k x_n} - \ee^{ 2 \pi \ii k x} \bigr\rvert + 2 \sum_{\lvert k \rvert > N } \lvert c _k \rvert\\
& \leq \sum_{k = - N }^{ N } \lvert c _k \rvert \bigl\lvert \ee^{ 2 \pi \ii k x_n} - \ee^{ 2 \pi \ii k x} \bigr\rvert + 4 M \sum_{k =N+1}^{ \infty } k^{- \beta },
\end{align*}
where we used our growth condition on the Fourier coefficents of $f$ in the last line. Now, if we let $n \rightarrow \infty$, we have
\begin{equation*}
\lim_{n \rightarrow \infty } \lvert f(x_n) - f(x) \rvert \leq 4 M \sum_{k= N+1 }^{ \infty } k^{- \beta }.
\end{equation*}
For $N \rightarrow \infty$, the right-hand side tends to zero, since we assumed that $\beta > 1$.
\end{rem}
\begin{proof}[Proof of Proposition~\ref{PropFourierMeth}]
We consider the sequence of polynomials $(f_m)_{m \in \N}$ with
\begin{equation*}
f_m(x) := \sum_{k=-m}^{m} c_k \ee^{2 \pi \ii k x}.
\end{equation*}
In the proof of Theorem~\ref{MainResThmLargeGaps} we saw that the LDP holds for $S_n^m := \sum_{ k=1}^n f_m( a_k U )$, $n\in\N$, with the i.i.d.\ rate function $ \widetilde{I}^{f_m}$. Now fix $\varepsilon > 0$ and choose $m=m( \varepsilon )$ such that \(\lim_{\varepsilon \to 0} m(\varepsilon) = \infty\) and
\begin{align*}
\lvert f(x)-f_m(x) \rvert & = \Bigl\lvert \sum_{ \lvert k \rvert > m} c_k \ee^{2 \pi \ii k x} \Bigr\rvert \\
& \leq \sum_{ \lvert k \rvert > m} \lvert c_k \rvert\\
& \leq \varepsilon.
\end{align*}
This approximation can be used to bound the G\"artner--Ellis limit from above: we get
\begin{equation*}
\frac{1}{n} \log\Exp\bigl[ \ee^{ \theta S_n } \bigr] = \frac{1}{n} \log\Exp\bigl[ \ee^{ \theta( S_n - S_n^m)}  \ee^{ \theta S_n^m} \bigr] \leq \varepsilon + \frac{1}{n} \log\Exp\bigl[ \ee^{ \theta S_n^m } \bigr].
\end{equation*}
Thus, taking the limes superior on both sides yields
\begin{equation*}
\limsup_{n \rightarrow \infty } \frac{1}{n} \log\Exp\bigl[ \ee^{ \theta S_n }  \bigr] \leq \varepsilon + \widetilde{ \Lambda}^{ f_m }( \theta ).
\end{equation*}
Similarly, one receives the lower bound 
\begin{equation*}
\liminf_{n \rightarrow \infty } \frac{1}{n} \Exp\bigl[ \ee^{ \theta S_n } \bigr] \geq \widetilde{ \Lambda}^{f_m}( \theta ) - \varepsilon. 
\end{equation*}
Now letting $\varepsilon \rightarrow 0$ (and hence $m \rightarrow \infty $) we get
\begin{equation*}
\liminf_{n \rightarrow \infty } \frac{1}{n} \Exp\bigl[ \ee^{ \theta S_n } \bigr] = \limsup_{n \rightarrow \infty } \frac{1}{n} \Exp\bigl[ \ee^{ \theta S_n } \bigr] = \lim_{m \rightarrow \infty } \widetilde{\Lambda}^{f_m}( \theta).
\end{equation*}
We note that the convergence of $(f_m)_{m \in \N}$ towards $f$ is uniform and that such an $f$ is bounded. This allows us to interchange integral and limit, and we end up with
\begin{equation*}
\lim_{n \to \infty} \frac{1}{n} \log\Exp\bigl[ \ee^{ \theta S_n } \bigr] = \lim_{m \rightarrow \infty } \widetilde{\Lambda}^{f_m}( \theta ) = \log \int_{0}^{1} \ee^{ \theta f(x)} \,\dif x .
\end{equation*}
The latter function is differentiable in $\theta$ and thus by the G\"artner--Ellis theorem we obtain an LDP at speed $n$ for $( S_n /n)_{n \in \N}$ with GRF $ \widetilde{I}^f$. 
\end{proof}

\begin{proof}[Proof of Theorem~\ref{MainResThmGeoProg}]
The proof follows essentially the same steps as the proof of Theorem~B and Proposition~3.4 in \cite{agkpr2020}, where the LDP is proven using tools from hyperbolic dynamics and mixing processes. For more information, we refer the reader to the references given on p.~19 in \cite{agkpr2020}. First, we define the map $\mathcal{T} \colon [0,1] \rightarrow [0,1]$ with
\begin{equation*}
\mathcal{T}( \omega ) := q \omega  - \lfloor q \omega  \rfloor.
\end{equation*}
Then, using that $a_k = q^k$ for $k \in \N $, the lacunary sum $S_n$ from \eqref{IntroEqLacSum} can be written as 
\begin{equation*}
S_n( \omega ) = \sum_{k=1}^{n} f( q^k \omega ) = \sum_{k=1}^{n} f( \mathcal{T}^k ( \omega )) , \quad \omega \in [0,1].
\end{equation*}
As in the proof of Theorem~\ref{MainResThmLargeGaps}, we use the G\"artner--Ellis theorem to show the LDP. Thus, we need to prove that the limit $\Lambda_q^f( \theta ) := \lim_{n \rightarrow \infty } \frac{1}{n} \log \Exp\bigl[ \ee^{ \theta S_n } \bigr] $ exists for all $\theta \in \R $ and is differentiable in $\theta $. In order to do so, we express $\ee^{\theta S_n}$ in terms of a certain linear operator. 
\par{}
Let $\Lip[0,1]$ be the Banach space of Lipschitz-continuous functions $g \colon [0,1] \rightarrow \C $, endowed with the norm $\lVert g \rVert := \lVert g \rVert_{ \infty } + L( g)$, where $L(g) $ is the Lipschitz constant of $g$. Then, for fixed $\theta \in \R$, define the linear operator $\Phi_{ \theta, q } \colon \Lip[0,1] \rightarrow \Lip[0,1]$ by
\begin{equation}
\label{EqPertubatedPFOperator}
\Phi_{ \theta, q }[g]( \omega ) :=  \frac{1}{q} \sum_{k=0}^{q-1} \ee^{ \theta f \bigl( \frac{ \omega + k}{q} \bigr) }g \Bigl( \frac{ \omega + k}{q} \Bigr), \quad \omega \in [0,1], \quad g \in \Lip[0,1].
\end{equation}
Next, we consider the Perron--Frobenius operator associated to $\mathcal{T}$, i.e., $\Phi_q \colon \Lip[0,1] \rightarrow \Lip[0,1]$ with
\begin{equation}
\label{EqPFOperator}
\Phi_{  q }[g]( \omega ) :=  \frac{1}{q} \sum_{k=0}^{q-1} g \Bigl( \frac{ \omega + k}{q} \Bigr), \quad \omega \in [0,1], \quad g \in \Lip[0,1],
\end{equation}
where we note that the operator from \eqref{EqPertubatedPFOperator} can be interpreted as a perturbation of the Perron--Frobenius operator in \eqref{EqPFOperator}. We have that
\begin{equation*}
\Phi_{ \theta, q }[g]( \omega ) = \Phi_q \bigl[ \ee^{\theta f } g \bigr](\omega) = \frac{1}{q} \sum_{k=0}^{q-1} \ee^{ \theta f \bigl( \frac{ \omega + k}{q} \bigr) }g \Bigl( \frac{ \omega + k}{q} \Bigr), \qquad \omega \in [0,1].
\end{equation*}
In a moment we will need the following basic property of $\Phi_q $: for $g \in \Lip[0,1]$,
\begin{equation}
\label{EqMeasPresPF}
\int_{0}^{1} \Phi_q[g]( \omega ) \,\dif\omega = \frac{1}{q} \sum_{k=0}^{q-1}  \int_{0}^{1} g \Bigl( \frac{ \omega + k}{q} \Bigr) \,\dif\omega = \frac{1}{q} \sum_{k=0}^{q-1}  \int_{k/q}^{ (k+1)/q} g (  x) q \,\dif x = \int_{0}^{1} g ( x) \,\dif x,
\end{equation} 
where we used the variable substitution $ ( \omega +k)/ q = x$ for $k=0,\dotsc,q-1$. 
Let $\Phi_{ \theta, q}^n$ and $\Phi_q^n$ denote the $n$-fold composition of $\Phi_{ \theta, q}$ and $\Phi_q$ respectively. Then, by Proposition~5.1 in \cite{baladi}, we have for $g \in \Lip[0,1]$
\begin{equation*}
\Phi_{ \theta, q}^n[g] = \Phi_q^n \bigl[ \ee^{ \theta S_n } g \bigr], \quad \text{for every } n \in \N.
\end{equation*}
Using this, we can write 
\begin{equation}
\label{EqGEExp}
\Exp\bigl[ \ee^{ \theta S_n }  \bigr] = \int_{0}^{1 } \ee^{ \theta S_n ( \omega )} \,\dif\omega =\int_{0}^{1 } \Phi_q^n \bigl[ \ee^{ \theta S_n} \bigr]  ( \omega ) \,\dif\omega = \int_{0}^{1 } \Phi_{ \theta , q}^n[\mathbb{1}]  ( \omega ) \,\dif\omega,
\end{equation}
where $\mathbb{1}$ denotes the constant function with value $1$. The second equation in \eqref{EqGEExp} holds due to the calculation in \eqref{EqMeasPresPF}.
\par{}
By assumption, $f$ is Lipschitz-continuous and hence Theorem~4.1 in \cite{zinsmeister} and Theorem~1.5 in \cite{baladi} are applicable, therefore we get that $\Phi_{ \theta, q}$ has a positive eigenvalue $\lambda_{ \theta , q} =: \lambda_{ \theta }$ with multiplicity $1$ and all other eigenvalues of $ \Phi_{ \theta, q}$ have strictly smaller modulus than $ \lambda_{ \theta }$. We use the well-known decomposition 
\begin{equation}
\label{EqOrthDecPFOp}
\Phi_{ \theta, q} = \lambda_{ \theta } Q_{ \theta } + R_{ \theta },
\end{equation}
where $ Q_{ \theta }$ is a projection operator onto the line spanned by an eigenfunction $h_{ \theta } > 0$ associated to the eigenvalue $\lambda_{ \theta }$, and $R_{ \theta }$ is an operator whose spectral radius is strictly smaller than \(\lambda_\theta\) and which is orthogonal to \(Q_\theta\) in the sense \(R_\theta Q_\theta = Q_\theta R_\theta = 0\). Moreover, there exists a probability measure $ \mu_{ \theta }  $ on $[0,1]$ such that for all $g \in \Lip[0,1]$ we have that
\begin{equation*}
Q_{ \theta }[g] = h_{ \theta } \, \frac{\int_{0}^{1} g( \omega ) \,\dif\mu_{ \theta} (\omega)  }{ \int_{0}^{1} h_{ \theta  }( \omega ) \,\dif\mu_{ \theta} (\omega)   }.
\end{equation*} 
Using these quantities and the orthogonality of $R_{ \theta }$ and $Q_{ \theta}$ gives us for $g \in \Lip[0,1]$
\begin{equation*}
\Phi_{ \theta, q}^n[g] = \lambda_\theta^n Q_{ \theta }^n[g] + R_{ \theta }^n[g] = \lambda_\theta^n h_{ \theta } \, \frac{\int_{0}^{1}g( \omega ) d \mu_{ \theta} (\omega)}{\int_{0}^{1} h_{ \theta  }( \omega )  d \mu_{ \theta} (\omega)} + R_{ \theta }^n[g].
\end{equation*}
If we set $g = \mathbb{1}$, using \eqref{EqGEExp}, we obtain
\begin{equation}
\label{EqGEExpOperator}
\Exp\bigl[ \ee^{ \theta S_n }  \bigr] = \int_{0}^{1 }\Phi_{ \theta , q}^n[\mathbb{1}](\omega) \,\dif\omega = \lambda_\theta^n \, \frac{\int_{0}^{1} h_{ \theta  }( \omega ) \,\dif\omega  }{ \int_{0}^{1} h_{ \theta  }( \omega ) \,\dif\mu_{ \theta} (\omega)   } + \int_{0}^{1} R_{\theta }^n[\mathbb{1}]( \omega ) \,\dif\omega .
\end{equation}
Since the spectral radius of $R_{ \theta }$ is strictly smaller than $\lambda_{ \theta }$, we get 
\begin{equation}
\label{EqlimGE}
\lim_{ n \rightarrow  \infty } \frac{\Exp\bigl[ \ee^{ \theta S_n }  \bigr]}{ \lambda_{ \theta }^n} = \frac{\int_{0}^{1} h_{\theta}( \omega ) \,\dif\omega}{ \int_{0}^{1} h_{ \theta  }( \omega ) \,\dif\mu_{ \theta} (\omega)}.
\end{equation}
In particular, we have that
\begin{equation*}
\lim_{ n \rightarrow  \infty } \frac{1}{n} \log \Exp\bigl[  \ee^{ \theta S_n }  \bigr] = \log \lambda_{ \theta},
\end{equation*}
which proves the existence of the G\"artner--Ellis limit. 
\par{}
We now turn to the proof of the remaining assertions, which we claim (and justify below)
can be deduced from the perturbation theory of linear operators (see Chapter~7 in \cite{KatoPertTh1995}), in particular the Kato-Rellich theorem, as stated in Theorem~4.24 in \cite{zinsmeister}. Indeed, since the family of operators $\Phi_{ \theta, q} $ depends on $ \theta  \in \C$ in an analytic way (see Proposition~5.1 (P3) in \cite{BroiseTransDil1996} and Theorem~1.7 in \cite{KatoPertTh1995}), the decomposition \eqref{EqOrthDecPFOp} continues to hold in some neighborhood $\mathcal{D}$ of the real axis (with $\lambda_{ \theta } , h_{ \theta }$ and $ \mu_{\theta }$ becoming complex-valued), with $\lambda_{ \theta } \neq 0$ and $ \lambda_{ \theta }$ (as well as $h_{ \theta } , \mu_{ \theta }, R_{ \theta }$ ) being analytic on $\mathcal{D}$. Moreover, $\lvert \lambda_{ \theta } \rvert$ stays strictly greater than the spectral radius of $R_{ \theta } $ if $\mathcal{D}$ is sufficiently small, which, looking at \eqref{EqGEExpOperator}, shows that convergence in \eqref{EqlimGE} is uniform on compact subsets of $\mathcal{D}$.
\par{}
For the second statement of Theorem~\ref{MainResThmGeoProg}, we fix $\theta \in \R$ and let $q \rightarrow \infty $. We note that the operator $\Phi_{ \theta , q}$ from \eqref{EqPertubatedPFOperator} is a Riemann sum and converges to the corresponding Riemann integral. This implies that the sequence of operators $ \Phi_{ \theta, q} $ for $q \geq 2 $ converges to the operator
\begin{equation*}
\widetilde{ \Phi}_{ \theta }[g]( \omega ) = \int_{0}^{1} \ee^{ \theta f( x )} g( x) d x = \widetilde{ \lambda }_{ \theta} \, \frac{ \int_{0}^{1} \ee^{ \theta f( x)} g(x) dx}{ \int_{0}^{1} \ee^{ \theta f( x)} dx} \cdot \mathbb{1}( \omega ),
\end{equation*}
where $ \widetilde{\lambda}_{ \theta }:=\int_{0}^{1} \ee^{ \theta f( x)} dx = \ee^{ \widetilde{\Lambda}^f( \theta )}$ with the cumulant generating function $  \widetilde{\Lambda}^f$ defined in \eqref{MainResEqiidCumGenFct}. Thus, $ \widetilde{\lambda}_{ \theta }$ is the Perron--Frobenius eigenvalue of $ \widetilde{ \Phi}_{ \theta }$, since $\widetilde{ \Phi}_{ \theta }/ \widetilde{\lambda}_{ \theta } $ is a projection onto the line spanned by the function $\mathbb{1} $.

Now if $\theta \in \R$ stays constant and $q \rightarrow \infty $, we can view $ \Phi_{ \theta,q}$ as perturbation of $\widetilde{\Phi}_{ \theta }$. By perturbation theory (see e.g.~\cite{KatoPertTh1995}), we have the convergence of the Perron--Frobenius eigenvalues, that is, $ \lim_{ q \rightarrow \infty } \lambda_{ \theta , q} = \widetilde{ \lambda }_{ \theta }$ for every $\theta \in \R$. Taking the logarithm, we get $\lim_{ q \rightarrow \infty } \Lambda_{q}^f( \theta ) =  \widetilde{\Lambda}^f( \theta )$. Since
the involved functions are convex, the convergence in fact is uniform on compact intervals. By taking the Legendre--Fenchel transform, it follows that $ \lim_{ q \rightarrow \infty } I_q^f(x) = \widetilde{I}^f(x)$ locally uniformly on $(-1, 1)$. 
\end{proof}

\begin{proof}[Proof of Corollary~\ref{MainResKorGeoCase}] 
Let $ ( f_m )_{m \in \N} $ be a sequence of Lipschitz continuous functions with $ \lVert f- f_m \rVert_\infty \leq \frac{1}{m}$ for all $m \in \N$ (we can find such a sequence by the Stone--Weierstraß theorem). We show that the sequence $S_n^m := \sum_{k= 1}^n f_m( a_k U)$ for $m \in \N$ is a sequence of exponentially good approximations for $S_n$. For more details see, e.g., Section~4.2.2 in \cite{dezei1998}. Note that for all $ \delta > 0$ we have
\begin{equation*}
\lim_{m \rightarrow \infty } \limsup_{n \rightarrow \infty } \frac{1}{n} \log\Pro\Bigl[ \frac{1}{n} \lvert S_n - S_n^m  \rvert > \delta \Bigr] = - \infty ,
\end{equation*} 
since $\lvert S_n - S_n^m \rvert = \bigl\lvert \sum_{ k=1}^n (f(a_k U)- f_m(a_k U)) \bigr\rvert \leq \frac{n}{m} $ and hence $\Pro\bigl[ \frac{1}{n} \lvert S_n - S_n^m  \rvert > \delta \bigr] = 0$ for sufficiently large $m \in \N$.
Thus, by Theorem~4.2.16 in \cite{dezei1998}, $ (S_n/n)_{n \in \N } $ satisfies a weak LDP at speed $n$ with rate function $\Rate_q^f \colon \R \rightarrow [0, \infty ]$ given by
\begin{equation*}
\Rate_q^f(x) := \sup_{ \delta > 0} \liminf_{m \rightarrow \infty } \inf_{ z \in B_\delta(x)} I_q^{f_m}(z).
\end{equation*}
Here, $B_{\delta}(x)$ denotes the ball around $x$ with radius $\delta$, and $I_q^{f_m}$ is the GRF of the LDP for $(S_n^m/n)_{n \in \N}$. 

Since $f$ is a continuous and periodic function, it is bounded and we have that $(S_n/n)_{n \in \N }$ is exponentially tight (see e.g.\ Section~1.2 in \cite{dezei1998}). Hence, $( S_n/n)_{n \in \N}$ satisfies the full LDP and $\Rate_q^f$ is a GRF. 
\end{proof}

% % % % % % % % % % % % % % % % % %
\section{Some open problems}
% % % % % % % % % % % % % % % % % % 

In this final section we collect some open problems that seem to be out of reach with the approach chosen in this paper.

\begin{prob}
It would be interesting to gather more properties of the cumulant generating function $\Lambda_q^f$ from Theorem~\ref{MainResThmGeoProg}. Maybe one can extract similar properties as in Theorem~B in \cite{agkpr2020} for the cosine.
\end{prob}

\begin{prob}
The cumulant generating function $\Lambda_q^f$ in the geometric case is of particular interest. One could try to calculate the G\"artner--Ellis limit directly (i.e., without using Perron--Frobenius theory). First, one can approximate $\ee^{ \theta f(x)}$ uniformly by trigonometric polynomials $f_m(x) = \sum_{j=-m}^{m} c_j \ee^{2 \pi \ii j x}$. We can then work with 
\begin{align*}
\Exp\bigl[ \ee^{ \theta S_n^m} \bigr] & = \int_{0}^{1} \prod_{k=1}^{n} \sum_{j=-m}^{m} c_j \ee^{2 \pi \ii j q^k x} \,\dif x.
\end{align*}
This leads to the set of solutions of
\begin{equation}
\label{ProblemsDioEqGeoSeq}
j_1q+ \dotsb +j_nq^n=0,
\end{equation}
where each $ j_i \in \{ -m,\dotsc,m\} $. Equation \eqref{ProblemsDioEqGeoSeq} has many solutions (indeed the number of solutions grows very fast in $n$) and it seems that the typical solution consists of many powers of $q$. 
\end{prob}

\begin{prob}
We again consider a continuous and periodic function $f \colon \R \rightarrow \R $. In this case we work with a lacunary sequence $(a_k)_{k \in \N}$ with
\begin{equation*}
\frac{a_{k+1}}{a_k} \rightarrow \eta > 1,
\end{equation*}
where $\eta $ is some transcendental number. Using uniform approximation via trigonometric polynomials and calculating the G\"artner--Ellis limit leads to the set of solutions
\begin{equation*}
j_1 a_1 + \dotsb + j_n a_n = 0,
\end{equation*}
where all $j_i \in \{ -m,\dotsc, m \} $ for some fixed $m \in \N$. One could even work with the equation
\begin{equation}
\label{ProblemsDioEqTranszSeq}
j_{k_n} a_{k_n}+ \dotsb + j_n a_n = 0,
\end{equation}
where $(k_n)_{n \in \N}$ is a sequence of natural numbers with $k_n \ll n$, and again $j_i \in \{ -m,\dotsc,m \}$. 
\newline 
It seems that for sufficiently large $n$, there are no non-trivial solutions to \eqref{ProblemsDioEqTranszSeq}. In this case we would be able to establish an LDP for $(S_n/n)_{n \in \N  }$ with the same GRF $\widetilde{I}^f$ as in the independent case.
\end{prob}

\subsection*{Acknowledgement}
LF, MJ, and JP are supported by the Austrian Science Fund (FWF) Project P32405 \textit{Asymptotic geometric analysis and applications} of which JP is principal investigator. LF and JP are also supported by the FWF Project F5513-N26 which is a part of the Special Research Program Quasi-Monte Carlo Methods: Theory and Applications. This work is part of the Ph.D. theses of LF and MJ written under supervision of JP.

\bibliographystyle{plain}
\bibliography{lacunary_sums_new_bib}

\bigskip
\bigskip

\medskip

\small

\noindent \textsc{Lorenz Fr\"uhwirth:} Institute of Mathematics and Scientific Computing,
University of Graz, Heinrichstra{\ss}e 36, 8010 Graz, Austria

\noindent
\textit{E-mail:} \texttt{lorenz.fruehwirth@uni-graz.at}

\medskip

\small

\noindent \textsc{Michael Juhos:} Institute of Mathematics and Scientific Computing,
University of Graz, Heinrichstra{\ss}e 36, 8010 Graz, Austria

\noindent
\textit{E-mail:} \texttt{michael.juhos@uni-graz.at}

\medskip

\small

\noindent \textsc{Joscha Prochno:} Institute of Mathematics and Scientific Computing,
University of Graz, Heinrichstra{\ss}e 36, 8010 Graz, Austria

\noindent
\textit{E-mail:} \texttt{joscha.prochno@uni-graz.at}

\end{document}